\documentclass{article}

\usepackage{
	amsmath, 
	amsthm, 
	amssymb,
	fancyhdr, 
	setspace,
	tikz-cd,
}
\usepackage{upgreek}

\newcommand{\za}{\alpha}
\newcommand{\zb}{\beta}

\newcommand{\zd}{\delta}
\newcommand{\zg}{\gamma}

\usepackage{lscape}
\usepackage{color}
\usepackage[color,matrix,arrow]{xy}
\usepackage{setspace}
\xyoption{all}

\usepackage{avant}
\usepackage{amsmath, amssymb, amsfonts, color, tikz, bbm, txfonts, euscript}
\usepackage{amssymb}
\usepackage{enumerate}
\usepackage{bm}
\newtheorem{theorem}{Theorem}[section]
\newtheorem{lemma}[theorem]{Lemma}

\newtheorem{prop}[theorem]{Proposition}
\newtheorem{cor}[theorem]{Corollary}
\theoremstyle{definition}
\newtheorem{mydef}[theorem]{Definition}
\newtheorem{example}[theorem]{Example}

\begin{document}

\thispagestyle{empty}

\title{BONGARTZ $\tau$-COMPLEMENTS OVER SPLIT-BY-NILPOTENT EXTENSIONS}
\author{Stephen Zito\thanks{The author was supported by the University of Connecticut-Waterbury}}
        
\maketitle

\begin{abstract}
Let $C$ be a finite dimensional algebra with $B$ a split extension by a nilpotent bimodule $E$, and let $M$ be a $\tau_C$-rigid module with $U$ its Bongartz $\tau$-complement.  If the induced module, $M\otimes_CB$, is $\tau_B$-rigid, we give a necessary and sufficient condition for $U\otimes_CB$ to be its Bongartz $\tau$-complement in $\mathop{\text{mod}}B$.  If $M$ is $\tau_B$-rigid, we again provide a necessary and sufficient condition for $U\otimes_CB$ to be its Bongartz $\tau$-complement in $\mathop{\text{mod}}B$ .
\end{abstract}

\section{Introduction}
Let $C$ be a finite dimensional algebra over an algebraically closed field $k$.  By module is meant throughout a finitely generated right $C$-module.  Following $\cite{AIR}$, we call a $C$-module $M$ $\tau_C$-$\emph{rigid}$ if $\text{Hom}_C(M,\tau_CM)=0$ and $\tau_C$-$\emph{tilting}$ if the number of pairwise nonisomorphic  indecomposable summands of $M$ equals the number of pairwise nonisomorphic simple modules of $C$.  We also say $M$ is $\emph{almost complete}$ $\tau_C$-$\emph{tilting}$ if $|M|=|C|-1$.  It was shown in \cite{AIR} that, given any $\tau_C$-rigid module, there exists a $\tau_C$-rigid module $U$ such that $M\oplus U$ is a $\tau_C$-tilting module.  This module $U$ is called the $\emph{Bongartz $\tau$-complement}$ of $M$.  In this paper, we are interested in the problem of extending Bongartz $\tau$-complements.  More precisely, let $C$ and $B$ be two finite dimensional $k$-algebras such that there exists a split surjective algebra morphism $B\rightarrow C$ whose kernel $E$ is contained in the radical of $B$.  We then say $B$ is a split extension of $A$ by the nilpotent bimodule $E$.
\par  
 The module categories of $C$ and $B$ are related by the functor $-\otimes_CB$.  Assuming $M\otimes_CB$ is a $\tau_B$-rigid module, we ask under what conditions will $U\otimes_CB$ be the Bongartz $\tau$-complement of $M\otimes_CB$ in $\mathop{\text{mod}}B$.  Our first main result is the following theorem.
 \begin{theorem}
 Let B be a split extension of C by a nilpotent bimodule E, and let M be a $\tau_C$-rigid module with $U$ its Bongartz $\tau$-complement.  Suppose $M\otimes_CB$ is $\tau_B$-rigid.  Then $U\otimes_CB$ is the Bongartz $\tau$-complement in $\mathop{\text{mod}}B$ if and only if $\emph{Hom}_C(U\otimes_CE,\tau_CM)=0$.
\end{theorem} 
Our second main result concerns $M$ as a $\tau_B$-rigid module and its Bongartz $\tau$-complement in $\mathop{\text{mod}}B$.  Here, $(\tau_BM)_C$ denotes the $C$-module structure of $\tau_BM$.
\begin{theorem}
 Let B be a split extension of C by a nilpotent bimodule E, and let M be a $\tau_C$-rigid module with $U$ its Bongartz $\tau$-complement.  Suppose $M$ is $\tau_B$-rigid.  Then $U\otimes_CB$ is the Bongartz $\tau$-complement in $\mathop{\text{mod}}B$ if and only if $\emph{Hom}_C(U,(\tau_BM)_C)=0$.
\end{theorem}

We use freely and without further reference properties of the module categories and Auslander-Reiten sequences as can be found in \cite{ASS}.  For an algebra $C$, we denote by $\tau_C$ the Auslander-Reiten translation in $\mathop{\text{mod}}C$.
 
 \subsection{Split extensions and extensions of scalars}
We begin this section with the formal definition of a split extension.
\begin{mydef}
 Let $B$ and $C$ be two algebras.  We say $B$ is a $\emph{split extension}$ of $C$ by a nilpotent bimodule $E$ if there exists a short exact sequence of $B$-modules
 \[
 0\rightarrow E\rightarrow B\mathop{\rightleftarrows}^{\mathrm{\pi}}_{\mathrm{\sigma}} C\rightarrow 0
\]
where $\pi$ and $\sigma$ are algebra morphisms, such that $\pi\circ\sigma=1_C$, and $E=\ker\pi$ is nilpotent.  
\end{mydef}
 A useful way to study the module categories of $C$ and $B$ is a general construction via the tensor product, also know as $\emph{extension of scalars}$, that sends a $C$-module to a particular $B$-module.
 \begin{mydef}
 Let $C$ be a subalgebra of $B$ such that $1_C=1_B$, then
 \[
 -\otimes_CB:\mathop{\text{mod}}C\rightarrow\mathop{\text{mod}}B
 \]
 is called the $\emph{induction functor}$, and dually
 \[
 D(B\otimes_CD-):\mathop{\text{mod}}C\rightarrow\mathop{\text{mod}}B
 \]
 is called the $\emph{coinduction functor}$.  Moreover, given $M\in\mathop{\text{mod}}C$, the corresponding induced module is defined to be $M\otimes_CB$, and the coinduced module is defined to be $D(B\otimes_CDM)$.  
 \end{mydef}
 The next proposition shows a precise relationship between a given $C$-module and its image under the induction and coinduction functors.
 \begin{prop}$\emph{\cite[Proposition~3.6]{SS}}$.  
\label{SS, Proposition 3.6}
Suppose B is a split extension of C by a nilpotent bimodule E.  Then, for every $M\in\mathop{\emph{mod}}C$, there exists two short exact sequences of B-modules:
\begin{enumerate}
\item[$\emph{(a)}$] $0\rightarrow M\otimes_CE\rightarrow M\otimes_CB\rightarrow M\rightarrow 0$
\item[$\emph{(b)}$] $0\rightarrow M\rightarrow D(B\otimes_CDM)\rightarrow D(E\otimes_CDM)\rightarrow 0$
\end{enumerate}
\end{prop}
 Thus, as a $C$-module, $M\otimes_CB\cong M\oplus (M\otimes_CE)$.  Next, we state a result which shows the number of indecomposable projective and injective modules remains the same under a split extension.  Here, $D$ denotes the standard duality functor. 
 \begin{prop}$\emph{\cite[Proposition~3.4]{SS}}$.
 \label{SS, Proposition 3.4}
 Let $C$ be a subalgebra of $B$ such that $1_C=1_B$.  If e is an idempotent then
 \begin{enumerate}
 \item[$\emph{(a)}$] $(eC)\otimes_CB\cong eB$.
 \item[$\emph{(b)}$] $D(B\otimes_C Ce)\cong DBe$.
 \end{enumerate}
 \end{prop}
 In particular, the number of simple modules remains unchanged under a split extension.
 Next, we state a description of the Auslander-Reiten translation of an induced module.
 \begin{lemma}$\emph{\cite[Lemma~2.1]{AM}}$
 \label{AM1}
 For a $C$-module M, we have
 \[
 \tau_B(M\otimes_CB)\cong\emph{Hom}_C(_BB_C,\tau_CM)
 \]
 \end{lemma}
This lemma is important because it allows us to use an adjunction isomorphism.  
 \begin{lemma}
\label{Adjunct}
 Let $M$ be a $C$-module, $M\otimes_CB$ the induced module, and let $X$ be any $B$-module.  Then we have 
\[
\emph{Hom}_B(X,\tau_B(M\otimes_CB))\cong\emph{Hom}_B(X,\emph{Hom}_C(_BB_C,\tau_CM)\cong\emph{Hom}_C(X\otimes_BB_C,\tau_CM)
\]
and
\[
\emph{Hom}_B(M\otimes_CB,X)\cong\emph{Hom}_C(M,\emph{Hom}_B(_CB_B,X)).
\]
\end{lemma}
\begin{proof}
These isomorphisms follow from Lemma $\ref{AM1}$ and the adjunction isomorphism.
\end{proof}
We note that $\_\otimes_BB_C$ and $\text{Hom}_B(_CB_B,\_)$ are two expressions for the forgetful functor $\mathop{\text{mod}}B\rightarrow\mathop{\text{mod}}C$.
 We end with the following useful fact.
 \begin{lemma}$\emph{\cite[Corollary~1.2]{AZ}}$.
 \label{AR submodule}
 $\tau_B(M\otimes_CB)$ is a submodule of $\tau_BM$.
 \end{lemma}

 \subsection{$\tau$-rigid modules and Bongartz $\tau$-complements}
We begin this section with several results on $\tau$-rigid modules.  We start with a definition.
 \begin{mydef} Let $M$ be a $C$-module.  We define $\mathop{Gen} M$ to be the class of all modules $X$ in $\mathop{\text{mod}}C$ generated by $M$, that is, the modules $X$ such that there exists an integer $d\geq0$ and an epimorphism $M^d\rightarrow X$ of $C$-modules.  Here, $M^d$ is the direct sum of $d$ copies of $M$.  Dually, we define $\mathop{Cogen}M$ to be the class of all modules $Y$ in $\mathop{\text{mod}}C$ cogenerated  by $M$, that is, the modules $Y$ such that there exist an integer $d\geq0$ and a monomorphism $Y\rightarrow M^d$ of $C$-modules.
\end{mydef} 
 The following result provides a characterization of $\tau$-rigid modules.
 \begin{prop}$\emph{\cite[Proposition~5.8]{AS}}$.
 \label{GenM Result}
   For M and N in $\mathop{\emph{mod}}C$, $\emph{Hom}_C(M,\tau_CN)=0$ if and only if $\emph{Ext}_C^1(N,\mathop{\emph{Gen}}M)=0$.
 \end{prop}
To describe Bongartz $\tau$-complements, we will begin with the definition of a torsion class and torsion pair.
\begin{mydef} A pair of full subcategories $(\mathcal{T},\mathcal{F})$ of $\mathop{\text{mod}}C$ is called a $\emph{torsion pair}$ if the following conditions are satisfied:
   \begin{enumerate}
   \item[(a)] $\text{Hom}_C(M,N)=0$ for all $M\in\mathcal{T}$, $N\in\mathcal{F}.$
   \item[(b)] $\text{Hom}_C(M,-)|_\mathcal{F}=0$ implies $M\in\mathcal{T}.$
   \item[(c)] $\text{Hom}_C(-,N)|_\mathcal{T}=0$ implies $N\in\mathcal{F}.$
   \end{enumerate}
We call $\mathcal{T}$ and $\mathcal{F}$ a $\it{torsion~class}$ and $\it{torsionfree~class}$ respectively.
\end{mydef}
\begin{mydef}
Let $\mathcal{T}$ be a full subcategory of $\mathop{\text{mod}}C$ and $X\in\mathcal{T}$.  We say a $C$-module $X$ is $\text{Ext}$-$\it{projective}$ in $\mathcal{T}$ if $\text{Ext}_C^1(X,\mathcal{T})=0$.  We denote by $P(\mathcal{T})$ the direct sum of one copy of each indecomposable $\text{Ext}$-projective module in $\mathcal{T}$ up to isomorphism.
\end{mydef}
Given a torsion class $\mathcal{T}$, we need to know when a module $X\in\mathcal T$ is Ext-projective in $\mathcal{T}$.  We have the following two results.
\begin{prop}$\emph{\cite[Proposition~2.9]{AIR}}$
\label{Ext1}
Let $\mathcal{T}$ be a functorially finite torsion class and $M$ a $\tau_C$-rigid module.  Then $M\in\emph{add}P(\mathcal{T})$ if and only if $\emph{Gen}M\subseteq\mathcal{T}\subseteq~\!^{\perp}(\tau_CM)$.
\end{prop}
We note for torsion classes $\mathcal{T}$, being functorially finite is equivalent to the existence of $M$ in $\mathop{\text{mod}}C$ such that $\mathcal{T}=\text{Gen}M$.
\begin{prop}$\emph{\cite[Proposition~1.11]{ASS}}$
\label{Ext2}
Let $(\mathcal{T},\mathcal{F})$ be a torsion pair in $\mathop\emph{mod}C$ and $M\in\mathcal{T}$ an indecomposable $C$-module.  Then $M$ is $\emph{Ext}$-projective in $\mathcal{T}$ if and only if $\tau_CM\in\mathcal{F}$. 
\end{prop}

It was shown in $\cite{AIR}$ that, for every $\tau_C$-rigid module $M$, there exists a module $U$ such that $M\oplus U$ is $\tau_C$-tilting.  

\begin{theorem}$\emph{\cite[Theorem~2.10]{AIR}}$
\label{Complement}
A $\tau_C$-rigid $C$-module is a direct summand of some $\tau_C$-tilting $C$-module.
\end{theorem}

This module is called the Bongartz $\tau$-complement of $M$.  To give an explicit construction, we define 
 \[
 ^{\perp}(\tau_CM)=\{X\in\mathop{\text{mod}}C~|~\text{Hom}_C(X,\tau_CM)=0\}.  
 \] 
  It was shown in $\cite{AIR}$ that $^{\perp}(\tau_CM)$ forms a torsion class.
  \begin{lemma}$\emph{\cite[Lemma~2.11]{AIR}}$  
 \label{AIR, Lemma 2.11}
 For any $\tau_C$-rigid module $M$, we have a sincere functorially finite torsion class $^{\perp}(\tau_CM)$.  The corresponding torsionfree class is $\emph{Cogen}(\tau_CM)$ and $(^{\perp}(\tau_CM),\emph{Cogen}(\tau_CM))$ is a torsion pair.
 \end{lemma}
 
  Then $P(^{\perp}(\tau_CM))$ is a $\tau_C$-tilting module satisfying  $M\in\text{add}(P(^{\perp}(\tau_CM)))$.  Let $U$ be the direct sum of one copy of each indecomposable $\text{Ext}$-projective module in $^{\perp}(\tau_CM)$ up to isomorphism 
that does not belong to $\text{add}M$.  Then $M\oplus U$ is $\tau_C$-tilting and $U$ is the Bongartz $\tau$-complement of $M$.  We end this section with a needed result on the number of summands of a $\tau$-rigid module.  We say a module $M$ is $\emph{basic}$ if all indecomposable summands of $M$ are pairwise nonisomorphic.
 
 \begin{prop}$\emph{\cite[Proposition~1.3]{AIR}}$
 \label{summands}
 Any basic $\tau_C$-rigid $M$ satisfies $|M|\leq|C|$.
 \end{prop}
We will need the following characterization of a $\tau$-rigid module being $\tau$-tilting.
 \begin{theorem}$\emph{\cite[Theorem~2.12]{AIR}}$
 \label{AIR, 2.12}
  Let $M$ be a $\tau_C$-rigid module.  Then $M$ is $\tau_C$-tilting if and only if $^{\perp}(\tau_CM)=\emph{Gen}M$.
  \end{theorem}

 The following result provides a useful restriction on $X$ when $X$ is an indecomposable summand of a $\tau_C$-tilting module.
\begin{prop}$\emph{\cite[Proposition~2.22]{AIR}}$
\label{ind}
Let $T=X\oplus U$ be a basic $\tau$-tilting $C$-module, with $X$ indecomposable.  Then exactly one of $^{\perp}(\tau_CU)\subseteq~\!^{\perp}(\tau_CX)$ and $X\in\emph{Gen}U$ holds.
\end{prop}

We end this section with a theorem which provides a necessary and sufficient condition for an induced module to be $\tau$-rigid.  We note that a $\emph{partial tilting module}$ is a $\tau$-rigid module such that the projective dimension is less than or equal to one and a $\emph{tilting module}$ is a $\tau$-tilting module with the projective dimension less than or equal to one.
\begin{theorem}$\emph{\cite[Theorem~A]{AM}}$.
\label{AM, Main Result}
Let $B$ be a split extension of $C$ by the nilpotent bimodule $E$, and $T$ be a $C$-module.  Then $T\otimes_CB$ is a (partial) tilting $B$-module if and only if $T$ is a (partial) tilting $C$-module, $\emph{Hom}_C(T\otimes_CE,\tau_CT)=0$, and $\emph{Hom}_C(D(E),\tau_CT)=0$.
\end{theorem}
In $\cite{AM}$, it is shown $\text{Hom}_C(D(E),\tau_CT)=0$ guarantees $T\otimes_CB$ will have the correct projective dimension.  For our purposes, we only require $\text{Hom}_C(T\otimes_CE,\tau_CT)=0$ which guarantees $T\otimes_CB$ will be $\tau_B$-rigid.

\section{Main Results and Corollaries}
We begin with our first main result.  Throughout, we assume $B$ is a split extension of $C$ by a nilpotent bimodule $E$.  
\begin{theorem}
\label{main}
Let M be a $\tau_C$-rigid module with $U$ its Bongartz $\tau$-complement in $\mathop{\text{mod}}C$.  Suppose $M\otimes_CB$ is $\tau_B$-rigid.  Then $U\otimes_CB$ is the Bongartz $\tau$-complement in $\mathop{\emph{mod}}B$ if and only if $\emph{Hom}_C(U\otimes_CE,\tau_CM)=0$.
\end{theorem}
\begin{proof}
Suppose $U\otimes_CB$ is the Bongartz $\tau$-complement of $M\otimes_CB$ in $\mathop{\text{mod}}B$.  This implies $\text{Hom}_B(U\otimes_CB,\tau_B(M\otimes_CB))=0$.  Using lemma $\ref{Adjunct}$ and proposition $\ref{SS, Proposition 3.6}$, we have the following isomorphisms 
\[
\text{Hom}_B(U\otimes_CB,\tau_B(M\otimes_CB))\cong\text{Hom}_B(U\otimes_CB,\text{Hom}_C(_BB_C,\tau_CM))\cong
\]
\[
\text{Hom}_C(U\otimes_CB\otimes_BB_C,\tau_CM)\cong\text{Hom}_C(U\otimes_CB_C,\tau_CM)\cong
\]
\[
\text{Hom}_C(U\otimes_C(C\oplus E)_C,\tau_CM)\cong\text{Hom}_C(U\oplus (U\otimes_CE),\tau_CM)\cong
\]
\[
\text{Hom}_C(U,\tau_CM)\oplus\text{Hom}_C(U\otimes_CE,\tau_CM).
\]

We conclude that $\text{Hom}_C(U\otimes_CE,\tau_CM)=0$.
\par
Conversely, suppose $\text{Hom}_C(U\otimes_CE,\tau_CM)=0$.  Then $\text{Hom}_C(U\otimes_CE,\tau_CU)=0$ because $U$ is Ext-projective in $^{\perp}(\tau_CM)$ and proposition $\ref{Ext2}$ shows $\tau_CU$ is cogenerated by $\tau_CM$ since $\text{Cogen}(\tau_CM)$ is the corresponding torsionfree class by lemma $\ref{AIR, Lemma 2.11}$.  Thus, theorem $\ref{AM, Main Result}$ says $U\otimes_CB$ is $\tau_B$-rigid.  Using the above vector space isomorphisms, we see $\text{Hom}_B(U\otimes_CB,\tau_B(M\otimes_CB))=0$.  Next, we will show $U\otimes_CB$ is Ext-projective in $^\perp(\tau_B(M\otimes_CB))$.  By proposition $\ref{Ext1}$, we need to show that 
\[
\mathop{\text{Gen}}(U\otimes_CB)\subseteq~^\perp(\tau_B(M\otimes_CB))\subseteq~^\perp(\tau_B(U\otimes_CB)).
\]
The first containment is clear so let $X\in~^\perp(\tau_B(M\otimes_CB))$ but $X\not\in~^\perp(\tau_B(U\otimes_CB))$.  Using the above vector space isomorphisms, $\text{Hom}_C(X_C,\tau_CM)=0$ and $\text{Hom}_C(X_C,\tau_CU)\not=0$ where $X_C$ denotes the $C$-module structure of $X$.  Since proposition $\ref{Ext2}$ says $\tau_CU$ is cogenerated by $\tau_CM$, we have a contradiction.  Thus, $U\otimes_CB$ is Ext-projective in $^\perp(\tau_B(M\otimes_CB))$.  
\par
Lastly, we need to show $U\otimes_CB$ comprises all the indecomposable Ext-projective modules in $^{\perp}(\tau_B(M\otimes_CB))$ up to isomorphism not in add($M\otimes_CB$).  Suppose not and let $Y$ be the direct sum of all remaining Ext-projective modules in $^{\perp}(\tau_B(M\otimes_CB))$ up to isomorphism not in add$(M\otimes_CB)$.  Then $(U\otimes_CB)\oplus Y$ is the Bongartz $\tau$-complement of $M\otimes_CB$ in $\mathop{\text{mod}}B$.  Thus, $(M\otimes_CB)\oplus(U\otimes_CB)\oplus Y$ is a $\tau_B$-tilting module such that the number of pairwise nonisomorphic  indecomposable summands equals the number of pairwise nonisomorphic simple modules of $B$.  However, proposition $\ref{SS, Proposition 3.4}$ implies the number of pairwise nonisomorphic simple modules of $C$ and $B$ are equal.  Thus, we have the inequality $|(M\otimes_CB)\oplus(U\otimes_CB)\oplus Y|>|B|$ but this contradicts proposition $\ref{summands}$.  We conclude $Y$ must be $0$ and $U\otimes_CB$ is the Bongartz $\tau$-complement of $M\otimes_CB$ in $\mathop{\text{mod}}B$.
\end{proof}
Next, we present three corollaries.  Let $M$ be a $\tau_C$-rigid module with $U$ is Bongartz $\tau$-complement in $\mathop{\text{mod}}C$.  In the case $M\in\text{Gen}U$, we may drop the assumption that $M\otimes_CB$ be $\tau_B$-rigid.
\begin{cor}
\label{cor 1}
Suppose $M\in\emph{Gen}U$.  Then $M\otimes_CB$ is $\tau_B$-rigid with $U\otimes_CB$ its Bongartz $\tau$-complement in $\mathop{\text{mod}}B$ if and only if $\emph{Hom}_C(U\otimes_CE,\tau_CM)=0$.
\end{cor}
\begin{proof}
We only need to show $M\otimes_CB$ being $\tau_B$-rigid follows from the assumption $\text{Hom}_C(U\otimes_CE,\tau_CM)=0$.  The rest follows from theorem $\ref{main}$.  Since $M\in\text{Gen}U$, there exists an epimorphism $f:U^d\rightarrow M$ where $d\geq 0$.  The functor $\_\otimes_CE$ is right exact and applying to $f$ yields an epimorphism $f\otimes_C1_E:(U\otimes_CE)^d\rightarrow M\otimes_CE.$  Thus, $\text{Hom}_C(U\otimes_CE,\tau_CM)=0$ implies $\text{Hom}_C(M\otimes_CE,\tau_CM)=0$ which further implies $M\otimes_CB$ is $\tau_B$-rigid by theorem $\ref{AM, Main Result}$.  
\end{proof}
In the special case where $M$ is indecomposable and non-projective, we always have $M\in\text{Gen}U$.
\begin{cor}
Let $M$ be indecomposable and non-projective.  Then $M\otimes_CB$ is $\tau_B$-rigid with $U\otimes_CB$ its Bongartz $\tau$-complement in $\mathop{\emph{mod}}B$ if and only if $\emph{Hom}_C(U\otimes_CE,\tau_CM)=0$.
\end{cor}
\begin{proof}
We need to show $M\in\text{Gen}U$ and the result will follow from corollary $\ref{cor 1}$.  By proposition $\ref{ind}$ either $M\in\text{Gen}U$ or $^{\perp}(\tau_CU)\subseteq~\!^{\perp}(\tau_CM)$.  Assume
$^{\perp}(\tau_CU)\subseteq~\!^{\perp}(\tau_CM)$ is true.  Since $U$ is the Bongartz $\tau$-complement in $\mathop{\text{mod}}C$, we have $^{\perp}(\tau_CM)\subseteq~\!^{\perp}(\tau_CU)$ by proposition $\ref{Ext1}$.  Thus, $^{\perp}(\tau_CU)=~\!^{\perp}(\tau_CM)$.  Again, since $U$ is the Bongartz $\tau$-complement of $M$, we know $\tau_CU\in\text{Cogen}(\tau_CM)$.  Now, $\text{Gen}M\subseteq~\!^{\perp}(\tau_CM)=~\!^{\perp}(\tau_CU)$ and proposition $\ref{Ext1}$ implies $M$ is Ext-projective in $^{\perp}(\tau_CU)$.  Proposition $\ref{Ext2}$ gives $\tau_CM\in\text{Cogen}(\tau_CU)$.  Since $\tau_CU$ and $\tau_CM$ cogenerate each other, we conclude $\tau_CM\cong\tau_CU$.  This is only possible if both $\tau_CM$ and $\tau_CU$ are 0 which implies $M$ and $U$ are projective.  But we assumed $M$ is not projective and thus a contradiction.  We conclude $M\in\text{Gen}U$.   
\end{proof}
Next, we assume that $E\in\text{Gen}M$ when $E$ is viewed as a right $C$-module.
\begin{cor}
Let $E\in\emph{Gen}M$.  Then $M\otimes_CB$ is $\tau_B$-rigid with $U\otimes_CB$ its Bongartz $\tau_B$-complement.
\end{cor}
\begin{proof}
Since $E\in\text{Gen}M$, we have $\text{Hom}_C(E,\tau_CM)=0$.  Since $\tau_CU$ is cogenerated by $\tau_CM$ by proposition $\ref{Ext2}$, we also have $\text{Hom}_C(E,\tau_CU)=0$.  Using the adjunction isomorphism,
\[
 0=\text{Hom}_C(M,\text{Hom}_C(E,\tau_CM))\cong\text{Hom}_C(M\otimes_CE,\tau_CM).
 \]
By Theorem $\ref{AM, Main Result}$, $M\otimes_CB$ is $\tau_B$-rigid.  By the same reasoning, $\text{Hom}_C(U\otimes_CE,\tau_CM)$ and $\text{Hom}_C(U\otimes_CE,\tau_CU)$ are equal to 0.  The result now follows from Theorem $\ref{main}$.
\end{proof}
Our next proposition concerns almost complete $\tau$-tilting modules.
\begin{prop}
\label{main2}
Suppose $M$ is an almost complete $\tau_C$-titling module such that $M\oplus Y$ is $\tau_C$-tilting and $Y$ is not the Bongartz $\tau_C$-complement for some indecomposable $C$-module $Y$.  Suppose $M\otimes_CB$ is $\tau_B$-tilting.  Then $(M\otimes_CB)\oplus(Y\otimes_CB)$ is $\tau_B$-tilting if and only if $\emph{Hom}_C(M\otimes_CE,\tau_CY)=0$.
\end{prop}
\begin{proof}
Since $Y$ is indecomposable and not the Bongartz $\tau_C$-complement, we have $Y\in\text{Gen}M$ by proposition $\ref{ind}$.  Thus, there exists an epimorphism $f:M^d\rightarrow Y$ where $d\geq0$.  The functor $\_\otimes_CB$ is right exact and applying to $f$ yields an epimorphism $f\otimes_C1_E:(M\otimes_CB)^d\rightarrow Y\otimes_CB$.  Since $M\otimes_CB$ is $\tau_B$-rigid and $Y\otimes_CB\in\text{Gen}(M\otimes_CB)$, we have $\text{Hom}_B(Y\otimes_CB,\tau_B(M\otimes_CB))=0$.  Using lemma $\ref{Adjunct}$ and proposition $\ref{SS, Proposition 3.6}$, we have
\[   
\text{Hom}_B(M\otimes_CB,\tau_B(Y\otimes_CB))\cong \text{Hom}_C((M\otimes_CB)_C,\tau_CY)\cong 
\]
\[
\text{Hom}_C(M,\tau_CY)\oplus \text{Hom}_C(M\otimes_CE,\tau_CY).
\]
Thus, $\text{Hom}_C(M\otimes_CE,\tau_CY)=0$ if and only if $\text{Hom}_B(M\otimes_CB,\tau_B(Y\otimes_CB))=0$ and our statement follows.
\end{proof}

\section{$M$ as a $\tau$-rigid $B$-module}
In this section, we present several results concerning a $C$-module $M$ when $M$ is also a $\tau_B$-rigid module.  Throughout, we assume $B$ is a split extension of $C$ by a nilpotent bimodule $E$ and $M$ is a $\tau_C$-rigid module.  We begin with a sufficient condition for $M$ to be $\tau_B$-rigid.

\begin{prop}
\label{result}
If $\emph{Hom}_C(M\otimes_CE,\emph{Gen}M)=0$, then $M$ is $\tau_B$-rigid.
\end{prop}
\begin{proof}
By proposition $\ref{SS, Proposition 3.6}$, we have the following short exact sequence in $\mathop{\text{mod}}B$ 

\[
0\rightarrow M\otimes_CE\rightarrow M\otimes_CB\rightarrow M\rightarrow 0.  
\]
Applying $\text{Hom}_B(-,\text{Gen}M)$, we obtain an exact sequence
\[
\text{Hom}_B(M\otimes_CE,\text{Gen}M)\rightarrow \text{Ext}_B^1(M,\text{Gen}M)\rightarrow \text{Ext}_B^1(M\otimes_CB,\text{Gen}M).
\]
First, we wish to show $\text{Ext}_B^1(M\otimes_CB,\text{Gen}M)=0$.  We know from proposition $\ref{GenM Result}$ this is equivalent to $\text{Hom}_B(M,\tau_B(M\otimes_CB))=0$.  By lemma $ \ref{Adjunct}$ and the assumption that $M$ is $\tau_C$-rigid, $\text{Hom}_B(M,\tau(M\otimes_CB))\cong\text{Hom}_C(M,\tau_CM)=0$.  Next, we want to show $\text{Hom}_B(M\otimes_CE,\text{Gen}M)=0$.  By restriction of scalars, any non-zero morphism from $M\otimes_CE$ to $\text{Gen}M$ in $\mathop{\text{mod}}B$ would give a non-zero morphism in $\mathop{\text{mod}}C$, contrary to our assumption.  Thus, $\text{Hom}_B(M\otimes_CE,\text{Gen}M)=0$.  We conclude $\text{Ext}_B^1(M,\text{Gen}M)=0$ and proposition $\ref{GenM Result}$ implies $M$ is $\tau_B$-rigid.

\end{proof}
The next result determines precisely when $M\otimes_CB$ is Ext-projective in $^{\perp}(\tau_BM)$.  Recall, we denote the $C$-module structure of $\tau_BM$ by $(\tau_BM)_C$.
\begin{prop}
\label{result2}
Suppose $M$ is $\tau_B$-rigid.  Then $M\otimes_CB\in P(^{\perp}(\tau_BM))$ if and only if $\emph{Hom}_C(M,(\tau_BM)_C)=0.$
\end{prop}
\begin{proof}
Assume $M\otimes_CB\in P(^{\perp}(\tau_BM))$.  Then $\text{Hom}_B(M\otimes_CB,\tau_BM)=0$.  Using lemma $\ref{Adjunct}$, we have $\text{Hom}_B(M\otimes_CB,\tau_BM)\cong\text{Hom}_C(M,(\tau_BM)_C)=0$.  Next, assume $\text{Hom}_C(M,(\tau_BM)_C)=0$.  Again, lemma $\ref{Adjunct}$ gives  $\text{Hom}_B(M\otimes_CB,\tau_BM)=0$.  Thus, $M\otimes_CB\in~^{\perp}(\tau_BM)$ and we need to show $M\otimes_CB\in(P^{\perp}(\tau_BM))$.  We have $\tau_B(M\otimes_CB)\in\text{Cogen}(\tau_BM)$ by lemma $\ref{AR submodule}$ and proposition $\ref{Ext2}$ gives $M\otimes_CB$ is Ext-projective in $^{\perp}(\tau_BM)$.

\end{proof}
Suppose $U$ is the Bongartz $\tau$-complement of $M$ in $\mathop{\text{mod}}C$.  Our main result gives a necessary and sufficient condition for $U\otimes_CB$ to be the Bongartz $\tau$-complement of $M$ in $\mathop{\text{mod}}B$.
\begin{theorem}
\label{main3}
Suppose $M$ is $\tau_B$-rigid.  Then $U\otimes_CB$ is the Bongartz $\tau$-complement in $\mathop{\emph{mod}}B$ if and only if $\emph{Hom}_C(U,(\tau_BM)_C)=0$.
\end{theorem}
\begin{proof}
Assume $U\otimes_CB$ is the Bongartz $\tau$-complement of $M$.  Then $\text{Hom}_B(U\otimes_CB,\tau_BM)=0$ and lemma $\ref{Adjunct}$ gives $\text{Hom}_B(U\otimes_CB,\tau_BM)\cong\text{Hom}_C(U,(\tau_BM)_C)=0$.  Next, assume $\text{Hom}_C(U,(\tau_BM)_C)=0$.  Again, lemma $\ref{Adjunct}$ gives $\text{Hom}_B(U\otimes_CB,\tau_BM)=0$.  Thus, $U\otimes_CB\in~^{\perp}(\tau_BM)$ and we need to show $U\otimes_CB\in(P^{\perp}(\tau_BM))$.  Using proposition $\ref{Ext1}$, we need to show the following containments
\[
\text{Gen}(U\otimes_CB)\subseteq~^{\perp}(\tau_BM)\subseteq~^{\perp}(\tau_B(U\otimes_CB)).
\]
The first is clear so let $X\in~^{\perp}(\tau_BM)$.  We need to show $X\in~^{\perp}(\tau_B(U\otimes_CB))$.  If $X\notin~^\perp(\tau_B(U\otimes_CB))$, then lemma $\ref{Adjunct}$ implies $\text{Hom}_B(X,\tau_B(U\otimes_CB))\cong\text{Hom}_C(X_C,\tau_CU)\not=0$.  Since $\tau_CU\in\text{Cogen}(\tau_CM)$, we would have $\text{Hom}_C(X_C,\tau_CM)\not=0$.  Since we assumed $X\in~^{\perp}(\tau_BM)$ and $\tau_B(M\otimes_CB)\in\text{Cogen}(\tau_BM)$ by lemma $\ref{AR submodule}$, we must have $\text{Hom}_B(X,\tau_B(M\otimes_CB))=0$.  However, using lemma $\ref{Adjunct}$, we see $\text{Hom}_B(X,\tau_B(M\otimes_CB))\cong\text{Hom}_C(X_C,\tau_CM)=0$, a contradiction.  Thus, we must have $X\in~^{\perp}(\tau_B(U\otimes_CB))$ and conclude by proposition $\ref{Ext1}$ that  $U\otimes_CB\in(P^{\perp}(\tau_BM))$.  Finally, to show $U\otimes_CB$ comprises all the indecomposable Ext-projective modules in $^{\perp}(\tau_BM)$ up to isomorphism not in add$M$, we apply the same reasoning used in the conclusion of theorem $\ref{main}$

\end{proof}
Our last result show that $(M\otimes_CB)\oplus (U\otimes_CB)$ and $M\oplus U$ are both $\tau_B$-tilting if and only if they are isomorphic to each other.
\begin{prop}
$M\oplus U$ and $(M\otimes_CB)\oplus (U\otimes_CB)$ are both $\tau_B$-tilting if and only if $M\otimes_CE=0$ and $U\otimes_CE=0$.

\end{prop}
\begin{proof}
Assume $M\oplus U$ and $(M\otimes_CB)\oplus (U\otimes_CB)$ are both $\tau_B$-tilting.  Since $M\otimes_BU$ is $\tau_B$-tilting, we know $\text{Ext}_B^1(M\oplus U, \text{Gen}(M\oplus U))=0$ by proposition $\ref{GenM Result}$.  Since $(M\otimes_CB)\oplus (U\otimes_CB)$ is $\tau_B$-tilting, we know $\text{Hom}_C((M\otimes_CE)\oplus(U\otimes_CE),\tau_C(M\oplus U))=0$ by theorems $\ref{AM, Main Result}$ and $\ref{main}$.  Thus, $(M\otimes_CE)\oplus(U\otimes_CE)\in\text{Gen}(M\oplus U)$ by theorem $\ref{AIR, 2.12}$.  However, $\text{Ext}_B^1(M\oplus U, (M\otimes_CE)\oplus(U\otimes_CE))\neq0$ by proposition $\ref{SS, Proposition 3.6}$.  This contradicts $\text{Ext}_B^1(M\oplus U, \text{Gen}(M\oplus U))=0$ unless $M\otimes_CE$ and $U\otimes_CE$ are equal to $0$.
\par 
Assume 
$M\otimes_CE$ and $U\otimes_CE$ are equal to $0$.  Proposition $\ref{SS, Proposition 3.6}$ implies $(M\otimes_CB)\oplus (U\otimes_CB)\cong(M\oplus U)$.  Also,  $\text{Hom}_C((M\otimes_CE)\oplus(U\otimes_CE),\tau_C(M\oplus U))=0$ implies $(M\otimes_CB)\oplus (U\otimes_CB)$ is $\tau_B$-tilting by theorems $\ref{AM, Main Result}$ and $\ref{main}$ and our statement follows.

\end{proof}

\section{Example}
In this section we give two examples illustrating our results.  We will construct a cluster-tilted algebra from a tilted algebra.  Such a construction is an example of a split extension.  Let $A$ be the path algebra of the following quiver:
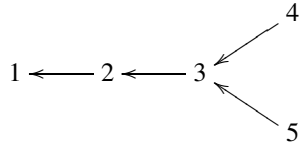
\[
\xymatrix@R10pt{&&&4\ar[dl]\\1&2\ar[l]&3\ar[l]\\ &&&5\ar[ul]}
\]
\par
Since $A$ is a hereditary algebra, we may construct a tilted algebra.  To do this, we need an $A$-module which is tilting.  Consider the Auslander-Reiten quiver of $A$ which is given by:
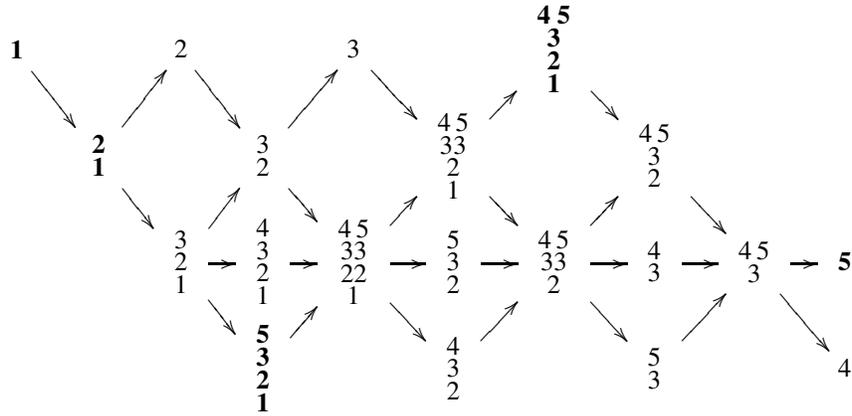
\[
\xymatrix@C=10pt@R=0pt{
{\begin{array}{c} \bf 1\end{array}}\ar[dr]&&
{\begin{array}{c} 2\end{array}}\ar[dr]&&
{\begin{array}{c} 3\end{array}}\ar[dr]&&
{\begin{array}{c}  \bf 4\ 5\\  \bf 3\\ \bf 2\\ \bf 1 \end{array}}\ar[dr]&&
\\
&{\begin{array}{c}  \bf 2\\ \bf 1\end{array}}\ar[dr]\ar[ur]&&
{\begin{array}{c} 3\\2\end{array}}\ar[dr]\ar[ur]&&
{\begin{array}{c} 4\,5\\33\\2\\1\end{array}}\ar[dr]\ar[ur]&&
{\begin{array}{c} 4\,5\\3\\2\end{array}}\ar[dr]&&
\\
&&{\begin{array}{c} 3\\2\\1\end{array}}\ar[dr]\ar[ur]\ar[r]&
{\begin{array}{c} 4\\3\\2\\1\end{array}}\ar[r]&
{\begin{array}{c} 4\,5\\33\\22\\1\end{array}}\ar[dr]\ar[ur]\ar[r]&
{\begin{array}{c} 5\\3\\2\end{array}}\ar[r]&
{\begin{array}{c} 4\,5\\33\\2\end{array}}\ar[dr]\ar[ur]\ar[r]&
{\begin{array}{c} 4\\3\end{array}}\ar[r]&
{\begin{array}{c} 4\,5\\3\end{array}}\ar[dr]\ar[r]&
{\begin{array}{c}  \bf 5\end{array}}
\\
&&&{\begin{array}{c} \bf  5\\ \bf 3\\ \bf 2\\ \bf 1\end{array}}\ar[ur]&&
{\begin{array}{c} 4\\3\\2\end{array}}\ar[ur]&&
{\begin{array}{c} 5\\3\end{array}}\ar[ur]&&
{\begin{array}{c} 4\end{array}}
}
\]

Let $T$ be the tilting $A$-module
\[
T=
{\begin{array}{c} 5\end{array}} \oplus
{\begin{array}{c} 4\,5\\3\\2\\1\end{array}} \oplus
{\begin{array}{c} 5\\3\\2\\1\end{array}} \oplus
{\begin{array}{c} 2\\1\end{array}} \oplus
{\begin{array}{c} 1\end{array}} 
\]

The corresponding titled algebra $C=\text{End}_AT$ is given by the bound quiver
$$\begin{array}{cc}
\xymatrix{1\ar[r]^\za&2\ar[r]^\zb&3\ar[r]^\zg&4\ar[r]&5}
&\quad\za\zb\zg=0\end{array}$$

Then, the Auslander-Reiten quiver of $C$ is given by: 

$$
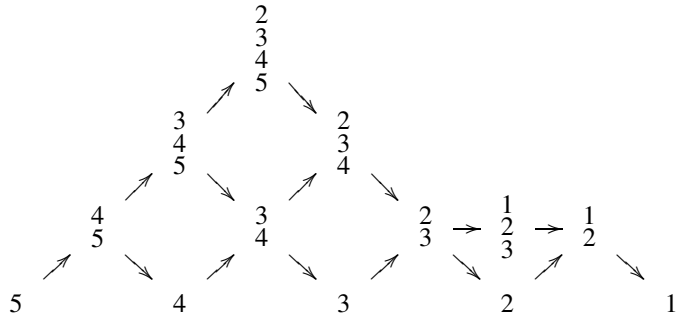
\xymatrix@C=10pt@R=0pt
{&&&{\begin{array}{c} 2\\ 3\\4\\5 \end{array}}\ar[dr]&&
&&
\\
&&{\begin{array}{c} 3\\4\\5 \end{array}}\ar[ur]\ar[dr]&&
{\begin{array}{c} 2\\ 3\\4 \end{array}}\ar[dr]&
\\
&{\begin{array}{c} 4\\5  \end{array}}\ar[ur]\ar[dr]&&
{\begin{array}{c} 3\\ 4  \end{array}}\ar[ur]\ar[dr]&&
{\begin{array}{c} 2\\3  \end{array}}\ar[r]\ar[dr]&
{\begin{array}{c} 1\\2\\ 3 \end{array}}\ar[r]&
{\begin{array}{c} 1\\2 \end{array}}\ar[dr]&&
\\
{\begin{array}{c}\ \\5\\ \  \end{array}}\ar[ur]&&
{\begin{array}{c}\ \\ 4\\ \  \end{array}}\ar[ur]&&
{\begin{array}{c}\ \\ 3\\ \  \end{array}}\ar[ur]&&
{\begin{array}{c}\ \\ 2\\ \  \end{array}}\ar[ur]&&
{\begin{array}{c}\ \\ 1\\ \  \end{array}}&&
}$$

The corresponding cluster-tilted algebra $B=C\ltimes\text{Ext}_C^2(DC,C)$ is given by the bound quiver 
$$\begin{array}{cc}
\xymatrix{1\ar[r]^\za&2\ar[r]^\zb&3\ar[r]^\zg&4\ar@/^20pt/[lll]^\zd\ar[r]&5}
&\quad \za\zb\zg=\zb\zg\zd=\zg\zd\za=\zd\za\zb=0 
\end{array}$$

Then, the Auslander-Retien quiver of $B$ is given by:

$$
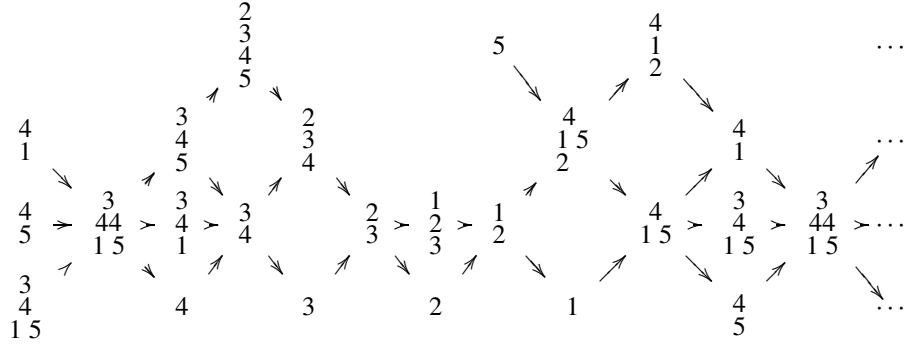
\xymatrix@C=3pt@R=0pt
{&&{\begin{array}{c} \end{array}}
&&{\begin{array}{c} 2\\ 3\\4\\5 \end{array}}\ar[dr]&&
{\begin{array}{c} \end{array}}&& 
{\begin{array}{c} 5 \end{array}}\ar[dr]&&
{\begin{array}{c} 4\\ 1\\2 \end{array}}\ar[dr]&&
{\begin{array}{c} \end{array}} &\cdots&
\\
&{\begin{array}{c} 4\\1 \end{array}}\ar[dr]&&
{\begin{array}{c} 3\\4\\5 \end{array}}\ar[ur]\ar[dr]&&
{\begin{array}{c} 2\\3\\4 \end{array}}\ar[dr]&&
{\begin{array}{c} \end{array}}&&
{\begin{array}{l} \ 4\\1\ 5\\2 \end{array}}\ar[ur]\ar[dr]&&
{\begin{array}{c} 4\\1 \end{array}}\ar[dr]&&
\cdots
\\
&{\begin{array}{c} 4\\ 5 \end{array}}\ar[r]&
{\begin{array}{c} 3\\44\\ 1\ 5 \end{array}}\ar[ur]\ar[dr]\ar[r]&
{\begin{array}{c} 3\\4\\1 \end{array}}\ar[r]&
{\begin{array}{c} 3\\4 \end{array}}\ar[ur]\ar[dr]&
{\begin{array}{c} \end{array}}&
{\begin{array}{c} 2\\ 3 \end{array}}\ar[dr]\ar[r]&
{\begin{array}{c} 1\\2\\ 3 \end{array}}\ar[r]&
{\begin{array}{c} 1\\2 \end{array}}\ar[dr]\ar[ur]&
{\begin{array}{c} \end{array}}&
{\begin{array}{c} 4\\1\ 5 \end{array}}\ar[dr]\ar[r]\ar[ur]&
{\begin{array}{c} 3\\4\\1\ 5 \end{array}}\ar[r]&
{\begin{array}{c} 3\\44\\ 1\ 5 \end{array}}\ar[ur]\ar[dr]\ar[r]&
\cdots
\\
&{\begin{array}{c}3 \\ 4\\ 1\ 5 \end{array}}\ar[ur]
&&{\begin{array}{c}\ \\ 4\\ \  \end{array}}\ar[ur]&&
{\begin{array}{c}\ \\ 3\\ \  \end{array}}\ar[ur]&&
{\begin{array}{c}\ \\ 2\\ \  \end{array}}\ar[ur]&&
{\begin{array}{c}\ \\ 1\\ \  \end{array}}\ar[ru]&&
{\begin{array}{c}\ \\ 4\\ 5  \end{array}}\ar[ur]&&
\cdots
}$$
\begin{example}
In $\mathop{\text{mod}}C$, consider $M={\begin{array}{c}2\\3\\4\\5\end{array}}\oplus{\begin{array}{c}2\\3\\4\end{array}}\oplus{\begin{array}{c}3\end{array}}$.  $M$ is a $\tau_C$-rigid module with Bongartz $\tau$-complement $U={\begin{array}{c}1\\2\\3\end{array}}\oplus{\begin{array}{c}3\\4\end{array}}$.  In this case, we have $M\otimes_CB\cong M$ which implies $M\otimes_CE=0$.  Thus, $M\otimes_CB\cong M$ is $\tau_B$-rigid and the induced module of $U$, $U\otimes_CB={\begin{array}{c}1\\2\\3\end{array}}\oplus{\begin{array}{c}3\\4\\1\end{array}}$, is the Bongartz $\tau$-complement in $\mathop{\text{mod}}B$.  Notice, we have $\tau_CM={\begin{array}{c}3\\4\\5\end{array}}\oplus{\begin{array}{c}4\end{array}}$, $U\otimes_CE={\begin{array}{c}1\end{array}}$, and $\text{Hom}_C(U\otimes_CE,\tau_CM)=0$, in accordance with theorem $\ref{main}$. 
\end{example}

\begin{example}
In $\mathop{\text{mod}}C$, consider $M={\begin{array}{c}3\\4\\5\end{array}}$.  $M$ is projective with Bongartz $\tau$-complement $U={\begin{array}{c}5\end{array}}\oplus{\begin{array}{c}4\\5\end{array}}\oplus{\begin{array}{c}2\\3\\4\\5\end{array}}\oplus{\begin{array}{c}1\\2\\3\end{array}}$. We have $M\otimes_CE=1$ and $\text{Hom}_C(M\otimes_EC,\text{Gen}M)=0$.  Thus, $M$ is $\tau_B$-rigid by proposition $\ref{result}$ with $\tau_BM={\begin{array}{c}4\\1\end{array}}$. Since  $M\otimes_CB={\begin{array}{c} 3\\4\\1\ 5 \end{array}}$,  proposition $\ref{result2}$ says $M\otimes_CB\in P(^{\perp}(\tau_BM))$ because $\text{Hom}_C(M,(\tau_BM)_C)=\text{Hom}_C(M,4\oplus 1)=0$. 
\par
~
\par
We have $U\otimes_CB={\begin{array}{c}5\end{array}}\oplus{\begin{array}{l} \ 4\\1\ 5\\2 \end{array}}\oplus{\begin{array}{c}2\\3\\4\\5\end{array}}\oplus{\begin{array}{c}1\\2\\3\end{array}}$.  Here, not every summand of $U\otimes_CB$ is a summand of the Bongartz $\tau$-complement in $\mathop{\text{mod}}B$.   because $\text{Hom}_B({\begin{array}{c}1\\2\\3\end{array}}\oplus{\begin{array}{l} \ 4\\1\ 5\\2 \end{array}},{\begin{array}{c}4\\1\end{array}})\not=0$.  Notice, $(\tau_BM)_C={\begin{array}{c}4\end{array}}\oplus{\begin{array}{c}1\end{array}}$ and $\text{Hom}_C({\begin{array}{c}1\\2\\3\end{array}}\oplus,{\begin{array}{c}4\end{array}}\oplus{\begin{array}{c}1\end{array}})\not=0$ in accordance with theorem $\ref{main3}$.  However, theorem $\ref{main3}$ guarantees ${\begin{array}{c}5\end{array}}\oplus{\begin{array}{c}2\\3\\4\\5\end{array}}$ are summands of the Bongartz $\tau$-complement in $\mathop{\text{mod}}B$  since $\text{Hom}_C({\begin{array}{c}5\end{array}}\oplus{\begin{array}{c}2\\3\\4\\5\end{array}}, 4\oplus 1)=0$.

\end{example}

\noindent Mathematics Faculty, University of Connecticut-Waterbury, Waterbury, CT 06702, USA
\it{E-mail address}: \bf{stephen.zito@uconn.edu}

\end{document}